\documentclass[11pt]{amsart}
\usepackage{amssymb, amsmath}
\usepackage{graphicx}
\usepackage{color}
\usepackage[all]{xy}
\usepackage{float}
\usepackage{etex}
\usepackage{hyperref}
\newtheorem{thm}{\bf Theorem}[section]

\newtheorem{prop}[thm]{\bf Proposition}
\newtheorem{definition}[thm]{\bf Definition}

\usepackage[dvipsnames]{xcolor}

\usepackage{color}

 	\definecolor{azure(colorwheel)}{rgb}{0.0, 0.5, 1.0}
  	\definecolor{awesome}{rgb}{1.0, 0.13, 0.32}

\begin{document}
\title[Folding maps on a cross-cap]
{Folding maps on a cross-cap\\
}
\author[M.~Barajas~S.]{Mart\'{i}n Barajas S.}
\address[M.~Barajas~S.]{Proyecto Curricular de Matem\'aticas, Facultad de Ciencias y Educaci\'on,
Universidad Distrital F.J.C., Bogot\'a, Colombia.}
\email{mbarajass@udistrital.edu.co}
%
%\author[Y.~Kabata]{Yutaro Kabata}
%\address[Y. ~Kabata]{Department of Mathematics,
%Graduate School of Science,  Hokkaido University,
%Sapporo 060-0810, Japan}
%\email{kabata@mail.sci.hokudai.ac.jp}
%
%
\subjclass[2010]{Primary 53A05; secondary 57R45}
\keywords{Crosscaps, folding maps, singularities}
\date{\today}
\begin{abstract}
We study the reflectional symmetry of a surface in the Euclidean $3$-dimensional space with a cross-cap singularity with respect to planes. 
This symmetry is picked up by the singularities of folding maps on the cross-cap. 
We give a list of the generic singularities that appear in the members of the family of folding maps on a cross-cap 
and characterise them geometrically. 
\end{abstract}
\maketitle
\setlength{\baselineskip}{14pt}
%
%%%%%%%%%%%%%%%%%%%%%%%%%%%%%%%%%%%%%%%%%%%%%%%%%%%%%%%%%%%%%%%%%%%%%%%%%%%
%%%%  Introduction
%%%%%%%%%%%%%%%%%%%%%%%%%%%%%%%%%%%%%%%%%%%%%%%%%%%%%%%%%%%%%%%%%%%%%%%%%%%
%
\section{Introduction}\label{section1}
Given a plane $W$ in the Euclidean 3-space $\mathbb{R}^3$ with normal vector $\eta$, the \textit{folding map} with respect to $W$ is the map $f_W:\mathbb{R}^3\rightarrow\mathbb{R}^3$ given by 
\begin{equation}\label{eqfold}
f_W(p)=q+\lambda^2 \eta,
\end{equation}
where $q$ is the projection of $p$ into $W$ along $\eta$ and $\lambda$ is the distance between $p$ and $W$. Thus, $p$ and its reflection with respect to $W$ have the same image under $f_W$.
Varying the plane $W$ gives the \textit{family of folding maps}.

Bruce and Wilkinson studied in \cite{Bruce1991,Wilkinson1991} the family of folding maps restricted to a smooth surface $M$ in  $\mathbb{R}^3$. A member of the family is locally a map-germ $\mathbb R^2,0\to\mathbb R^3,0$. They showed that the singularities types of the folding map, when allowing smooth changes of coordinates in the source and target, capture some aspect of the extrinsic geometry of the surface $M$.  
For instance, the folding map is singular at $p\in M$ if and only if $W$ is a member of the pencil of planes containing the normal line of $M$ at $p$. The singularity is more degenerate than a cross-cap if and only if $W$ is orthogonal to a principal direction of $M$ at $p$. A certain type of singularities of the folding maps ($S_2$) occurs along a curve on the surface called {\it the sub-parabolic curve}. It turns out that the sub-parabolic curve is the locus of points where a principal curvature is extremal along the lines of the other principal curvature.
Another singularity type ($B_2$) occurs on the {\it ridge curve} which 
is the locus of points where a principal curvature is extremal along its own lines of principal curvature.
Bruce and Wilkinson also proved that the bifurcation set of the family if folding maps is the dual of the union of the focal and symmetry sets of $M$. This captures, for instance, the parabolic set on the focal surface of $M$. An analogous work was carried out in \cite{Izumiya2010} for surfaces in the hyperbolic or the Sitter 3-dimensional spaces.

The cross-cap singularity occurs stably on parametrized surfaces in $\mathbb R^3$. 
The extrinsic differential geometry of the cross-cap was initiated by Bruce and West in \cite{BruceWest1998, West1995}. Since then, several work was carried out on the subject; see for example \cite{Barajas2017,BarajasKabata2016,Dias2016,Fukui2012,Garcia2000,Tari2007}. This paper is part of this ongoing work on the geometry of the cross-cap. 
We consider here the reflectional symmetry of a cross-cap with respect to planes and consider folding maps on the cross-cap. 
We obtain the list of generic singularities of these folding maps 
(Theorem \ref{th1}) and  describe their associated geometry (Theorem \ref{th2}). We give some preliminaries in Section \ref{section1} and  \ref{section2}
%%%%%%%%%%%%%%%%%%%%%%%%%%%%%%%%%%%%%%%%%%%%%%%%%%%%%%%%%%%%%%%%%%%%%%%%%%%
%%%%  Preliminaries
%%%%%%%%%%%%%%%%%%%%%%%%%%%%%%%%%%%%%%%%%%%%%%%%%%%%%%%%%%%%%%%%%%%%%%%%%%%
%
\section{Preliminaries}\label{section2}
\subsection{Geometric cross-cap}
Whitney showed that maps $\mathbb{R}^2\rightarrow\mathbb{R}^3$ can have a stable local singularity under smooth changes of coordinates in the source and target. A model of this singularity is given by $f(x,y)=(x,xy,y^2)$. 
A map germ $g:\mathbb{R}^2,0\rightarrow\mathbb{R}^3,0$ is said to have a cross-cap singularity if it is $\mathcal{A}$-equivalent to $f$ (we write  $f\sim_{\mathcal{A}}g$), that is, there exist germs of diffeomorphisms $h:\mathbb{R}^2,0\rightarrow\mathbb{R}^2,0$ and $k:\mathbb{R}^3,0\rightarrow\mathbb{R}^3,0$, such that $k\circ f=g\circ h$. 
The image of the map-germ $g$ is a germ of a singular surface called cross-cap. 

West showed in \cite{West1995} that a parametrization of a cross-cap can be taken, by a suitable choice of a coordinate system in the source and isometries in the target, in the form
\begin{equation}\label{Par_Cros}
\phi(x,y)=(x,xy+p(y),ax^2+bxy+y^2+q(x,y)),
\end{equation}
where $a$ and $b$ are constant real numbers. We write
$$
\begin{array}{rcl}
p(y)&=&p_3y^3+p_4y^4+O(5),\\
q(x,y)&=&\sum^{3}_{k=0}q_{3k}x^{3-k}y^k+\sum^{4}_{k=0}q_{4k}x^{4-k}y^k+O(5),
\end{array}
$$
where $O(l)$ denotes a reminder in $(x,y)$ of order $l$. 

A cross-cap with a parametrization as in \eqref{Par_Cros} is called \textit{geometric cross-cap} (in contrast to the $\mathcal A$-model $f$ above where the geometry is destroyed by the diffeomorphisms in the target). 

A key feature of the cross-cap is its double point curve. It is a regular curve 
in the source given by $x=-p_3y^2+\text{h.o.t.}$ for a cross-cap parametrized as in \eqref{Par_Cros} (see \cite{West1995}). This curve is mapped by $\phi$ to a segment of a curve ending at the cross-cap point, see Figure \ref{fig3}.

The tangent space to the cross-cap at the cross-cap point is one dimensional. We called this space the \textit{tangential line}. The orthogonal complement to the tangential line is called the \textit{normal plane}. The tangent cone at the cross-cap point is also an important object for the study of the differential geometry of the cross-cap. For a cross-cap with a parametrization as in \eqref{Par_Cros} 
the tangential line is parallel to the $x$-axis, the tangent cone coincides with the $xz$-plane and the normal plane is the $yz$-plane (see Figure \ref{fig3}).

\begin{figure}[h!!!]
 \includegraphics[width=7.5cm,clip]{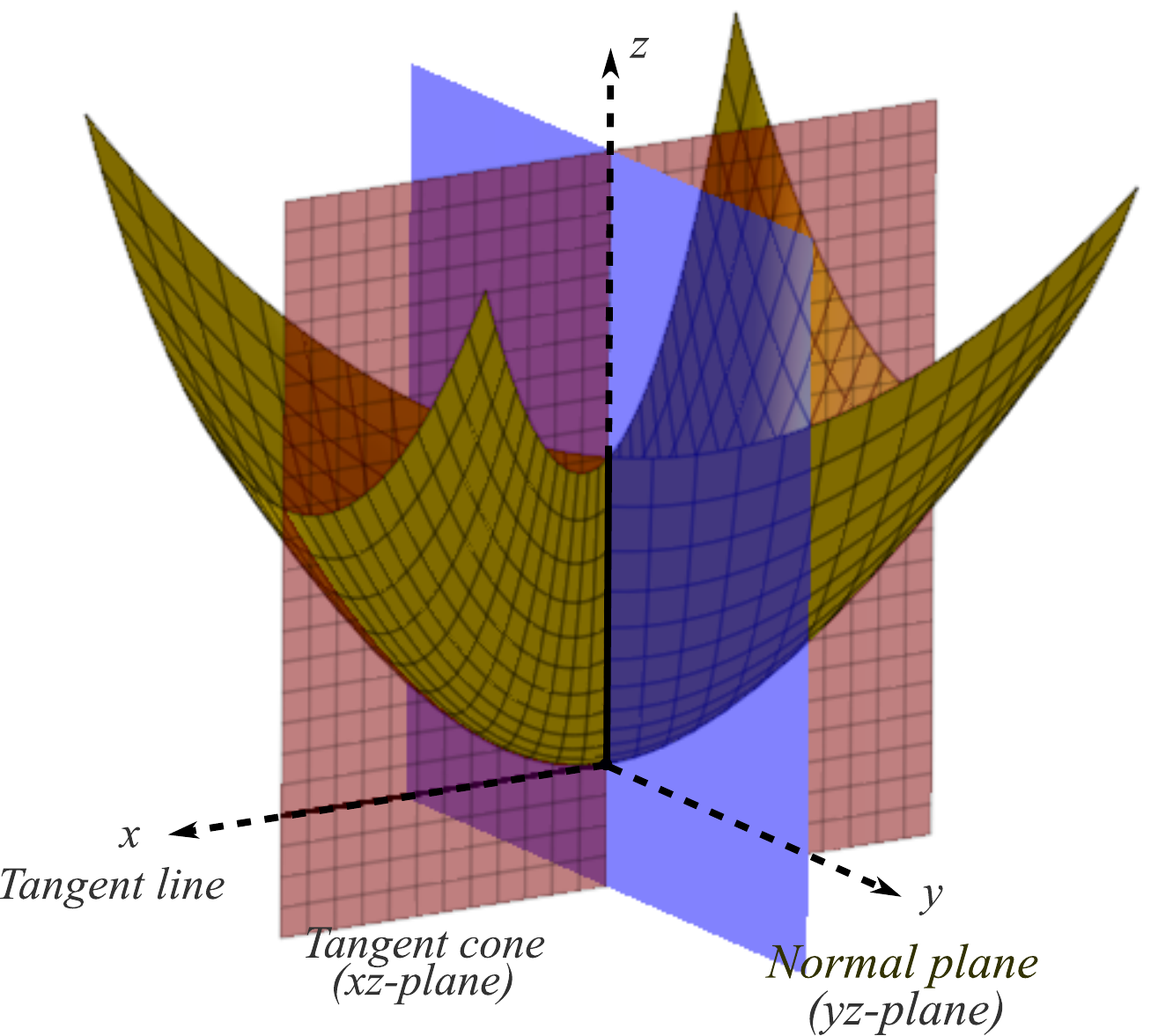}\\
 \caption{The tangential line, tangent cone and normal plane to an elliptic cross-cap.}
\label{fig3}
\end{figure}

The parabolic curve on a geometric cross-cap is studied in \cite{BruceWest1998,West1995}. When $a\ne 0$, the parabolic set 
in the source has a Morse singularity. If $a<0$, the parabolic set is an isolated point (the cross-cap point) and every regular point on the surface is hyperbolic. The cross-cap is called a hyperbolic cross-cap (Figure \ref{fig2}, right). 
If $a>0$, the parabolic set in the source is the union of two transverse curves and the cross-cap is called an elliptic cross-cap (Figure \ref{fig2}, left). When $a=0$, the cross-cap is called a parabolic cross-cap (Figure \ref{fig2}, center). For a generic parabolic cross-cap, the parabolic set in the source has a cusp singularity.  
We say that a geometric cross-cap is generic if $a\neq0$.

\begin{figure}[h!!!]
 \includegraphics[width=12.5cm,clip]{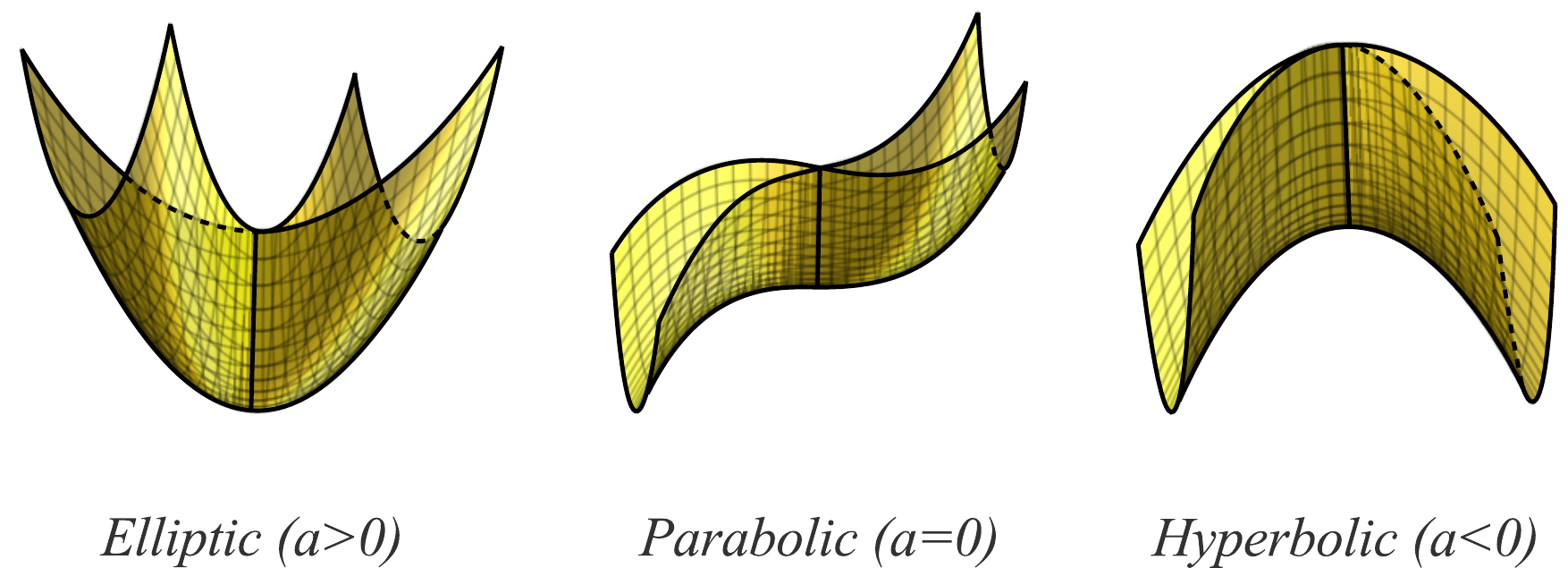}\\
 \caption{Geometric cross-caps (\cite{BruceWest1998,West1995}).}
\label{fig2}
\end{figure}

Let $p$ be a regular point on the cross-cap surface, and suppose that it is not an umbilic point. Let $v_i$ be a principal direction and $\kappa_i$ its associated principal curvature at $p$, for $i=1,2$. Denote by $v_i(\kappa_j)$ the directional derivative of $\kappa_j$ along $v_i$, $i,j=1,2$. The point $p$ is a sub-parabolic point relative to $v_i$ if $v_i(\kappa_j)(p)=0$, $i\neq j$. Similarly, $p$ is a ridge point relative to $v_i$ if $v_i(\kappa_i)(p)=0$.\\

%In \cite{West1995} was defined the limitant principal curvatures at the cross-cap point, more precisely:
%
\begin{prop}[\cite{West1995}]
Suppose that $C$ is a regular curve in the source that passes through the origin, parametrised by
$
\gamma(t)=\left(\alpha t+O(2),\beta t+O(2)\right).
$
Then as we approach the cross-cap point along the curve parametrised by $\phi\circ\gamma$ (for $\phi$ as in \eqref{Par_Cros}) one principal curvature tends to
\begin{equation}\label{k1cross-cap}
\frac{2\left(a\alpha^2-\beta^2\right)}{\alpha\left(\alpha^2+\left(2\beta+\alpha b\right)^2\right)^{\frac{1}{2}}}
\end{equation}
and the other tends to infinity.
\end{prop}
We denote by $\kappa_1$ the principal curvature that tends to \eqref{k1cross-cap} and by $\kappa_2$ the other one. The limiting principal vectors $\nu_i$ associate to $\kappa_i$ are defined in \cite{Fukui2012} by a natural limiting process. The ridge curve on a cross-cap is studied in \cite{Fukui2012} and \cite{West1995} and the sub-parabolic curve in \cite{Fukui2012}.\\

It is shown in \cite{Fukui2012} that there is at least one and at most three regular sub-parabolic curves relative to $\nu_2$ and there are no sub-parabolic points relative to $\nu_1$ at the cross-cap. Denote by $(w_1,w_2)$ the tangent directions of these curves in the source, then 
\begin{equation}\label{eq_subparab}
 abw_1^3+(b^2+2a+1)w_1^2w_2+3bw_1w_2^2+2w_2^3=0,
\end{equation}
where $a,b$ are the coefficients in the parametrisation \eqref{Par_Cros}. 
There are two ridge curves relative to $\nu_2$ at the cross-cap and their tangent directions $(w_1,w_2)$ in the source satisfy
\begin{equation}\label{eq_ridge1}
\left(bw_1-2w_2\right)w_1=0.
\end{equation}

In \cite{Tari2007} the author studied the lines of principal curvatures on a cross-cap and showed that there is generically one single topological model. In particular, there are three separatrices of the foliations, one is given by $y=-\frac{1}{2}q_{31}x^2+\text{h.o.t.}$ and the other two are given by $x=\lambda_i y^2+\text{h.o.t.}$, $i=1,2$, with
\begin{equation}\label{eq_sep_lc}
\lambda_i^2+3p_3\lambda_i-2=0.
\end{equation}

%%%%%%%%%%%%%%%%%%%%%%%%%%%%%%%%%%%%%%%%%%%%%%%%%%%%%%%%%%%%%%%%%%%%%%%%%%%%%%%%%%%%%%%%%
%
\section{The singularities of the folding maps on a cross-cap}\label{section3}
Consider a plane in $\mathbb R^3$ whose points $p$ satisfy 
\begin{equation*}
W_{(\eta,\delta)}: \left\langle p,\eta\right\rangle=\delta, 
\end{equation*}
with $\eta\in\mathbb{S}^2$ and $\delta\in\mathbb{R}$. 
Varying $\eta$ and $\delta$ gives all the planes in $\mathbb R^3$.
For a fixed geometric cross-cap $\phi$ as in \eqref{Par_Cros}, the family  $F:\mathbb{R}^2\times\mathbb{S}^2\times\mathbb{R}\rightarrow\mathbb{R}^3$ 
of folding maps on the cross-cap
is given by
\begin{eqnarray*}
F(x,y,\eta,\delta) & = & \phi(x,y)+(\left\langle \eta,\phi(x,y) \right\rangle-\delta)(\left\langle \eta,\phi(x,y) \right\rangle-\delta-1)\eta.
\end{eqnarray*}

We write $f_{(\eta,\delta)}$ for the restriction of $F$ to the plane $W_{(\eta,\delta)}$. 
This is locally a map-germ $\mathbb{R}^2,0\rightarrow\mathbb{R}^3,0$.
The singularities of such map-germs under the action of the group $\mathcal A$ of smooth changes of coordinates in the source and target are classified by Mond in \cite{Mond1982,Mond1985}. We determine here the $\mathcal A$-singularities of the map-germs $f_{(\eta,\delta)}$. 
Before that, we need some notation.\\

Let $\mathcal{E}_n$ denote the local ring of germs of functions $\mathbb{R}^n,0\rightarrow\mathbb{R}$ and $\mathcal{M}_n$ its maximal ideal. Denote by $\mathcal{E}(2,3)$ the 3-tuples of elements in $\mathcal{E}_2$.
The tangent space to the $\mathcal{A}$-orbit of $f: \mathbb{R}^2,0\rightarrow\mathbb{R}^3,0$ at the germ $f$ is given by
\begin{equation*}
T\mathcal{A}\cdot f=\mathcal{M}_2\cdot\left\{f_{x},f_{y}\right\}+f^*\left(\mathcal{M}_3\right)\cdot\left\{e_1,e_2,e_3\right\},
\end{equation*}
where $f_{x}$ and $f_y$ are the partial derivatives of $f$, $\left\{e_1,e_2,e_3\right\}$ denotes the standard basis vectors of $\mathbb{R}^3$ considered as elements of $\mathcal{E}(2,3)$ and $f^*\left(\mathcal{M}_3\right)$ is the pull-back of the maximal ideal in $\mathcal{E}_3$. \\

The extended tangent space is defined as 
\begin{equation*}
T_e\mathcal{A}\cdot f=\mathcal{E}_2\cdot\left\{f_{x},f_{y}\right\}+f^*\left(\mathcal{E}_3\right)\cdot\left\{e_1,e_2,e_3\right\},
\end{equation*}
and the $\mathcal{A}_e$-codimension of the germ $f$ is
\begin{equation*}
\mathcal{A}_e\text{-cod}\left(f\right)=\dim_{\mathbb{R}}\frac{\mathcal{E}(2,3)}{T_e\mathcal{A}\cdot f}.
\end{equation*}

Let $f\in\mathcal{M}_n\cdot\mathcal{E}(n,p)$. A $q$-parameter unfolding $(q,F)$ of $f$ is a map-germ
\begin{equation*}
F:\mathbb{R}^n\times\mathbb{R}^q,(0,0)\rightarrow\mathbb{R}^p\times\mathbb{R}^q,(0,0)
\end{equation*}
in the form $F(x,u)=\left(\bar{f}(x,u),u\right)$, with $\bar{f}(x,0)=f(x)$. The family 
\begin{equation*}
\bar{f}:\mathbb{R}^n\times\mathbb{R}^q,(0,0)\rightarrow\mathbb{R}^p,0\end{equation*}
is called a $q$-parameter deformation of $f$. Let $I$ be the identity element  in $\mathcal{A}$. A morphism between two unfoldings $(q_1,F)$ and $(q_2,G)$ is a pair $\left(\alpha,\psi\right):(q_1,F)\rightarrow(q_2,G)$ with $\alpha:\mathbb{R}^{q_1},0\rightarrow\mathcal{A},I$, $\psi:\mathbb{R}^{q_1},0\rightarrow\mathbb{R}^{q_2},0$, such that $\bar{f}_u=\alpha(u)\cdot\bar{g}_{\psi(u)}$. The unfolding $(q_1,F)$ is then said to be induced from $(q_2,G)$ by $\left(\alpha,\psi\right)$. An unfolding $(q_1,F)$ of a map-germ $f$ is said to be $\mathcal{A}_e$-versal if any unfolding $(q_2,G)$ of $f$ can be induced from $(q_1,F)$.\\

Now we state a fundamental theorem on unfoldings. Given an unfolding $F(x,u)=\left(f(x,u),u\right)$, the initial speeds, $\dot{F}_i\in\mathcal{E}(n,p)$, of $F$ are defined by
\begin{equation*}
\dot{F}_i(x)=\frac{\partial f}{\partial u_i}(x,0),\quad\text{for}\quad i=1,\ldots,q.
\end{equation*}
\begin{thm}[\cite{Martinet1982}]\label{thm_versal-unfolding}
	An unfolding $(q,F)$ of $f\in\mathcal{M}_n\cdot\mathcal{E}(n,p)$ is $\mathcal{A}_e$-versal if and only if
	\begin{equation*}
	T_e\mathcal{A}\cdot f+\mathbb{R}\cdot\left\{\dot{F}_1,\ldots,\dot{F}_q\right\}=\mathcal{E}(n,p).
	\end{equation*}
\end{thm}
Our goal in is to classify germs of folding maps on a cross-cap and to analise the deformations in the members of the family of folding 
maps. We start with the following result.
\begin{prop}\label{prop3.1}
For any $(\eta,\delta)\in\mathbb{S}^2\times\mathbb{R}$, 
the folding map $f_{(\eta,\delta)}$ on a cross-cap as in \eqref{Par_Cros} 
is singular at the origin. The 
singularity is more degenerate than a cross-cap if, and only if, the plane 
$W_{(\eta,\delta)}$ contains the origin, that is, $\delta=0$.
\end{prop}
\begin{proof}
The first part follows by observing that  $\frac{\partial f_{(\eta,\delta)}}{\partial y}(0,0)=(0,0,0)$ for all $\left(\eta,\delta\right)\in\mathbb{S}^2\times\mathbb{R}$, 
 thus \mbox{rank $df_{(\eta,\delta)}(0,0)\leq 1$}. For the second part we use Whitney's criteria for recognition of the cross-cap singularity (see \cite{Whitney1943}).
\end{proof}
We list below the non-stable $\mathcal A$-singularities that occur at the origin in the members of the family of folding maps on a cross-cap. By Proposition \ref{prop3.1} above, $\delta=0$. We denote $f_{(\eta,0)}$ by $f_\eta$ and write $\eta=(\alpha,\beta,\gamma)$.

%\textcolor{blue}{Falta explicar porque somente as singularidades listadas no Teorema 3.3 acontecem genericamente.}

\begin{thm}\label{th1}
For a generic cross-cap $\phi$ as in \eqref{Par_Cros}, the folding map $f_\eta$ has a singularity $\mathcal{A}$-equivalent to one in Table \ref{table1} when $\eta$ is transverse to the tangent cone, and to one in Table \ref{table2} when $\eta$ is in the tangent cone but is not parallel to the tangential direction.\\
When $\eta$ is parallel to the tangential direction, 
$f_\eta$ has corank 2 and is $\mathcal A$-equivalent to
\begin{equation}\label{eq_corank2}
(x^2,xy+y^3,y^2+Ax^3+Bx^2y+Cxy^2+y^3),
\end{equation}
with $\Theta(A,B,C)\neq0$, where 
%{ \scriptsize%\tiny%\footnotesize
\begin{eqnarray*}
\Theta(A,B,C) & = & -10240A^3+\left(-2560C^2-3840B-1440C-135\right)A^2+\\
							& + & \left(-160C^4+2880B^2C+800BC^2-20C^3+60B^2+\right.\\
							& + & \left. 90BC\right)A+40B^2C^3-540B^4-180B^3C+5B^2C^2-20B^3.
\end{eqnarray*}
%\begin{eqnarray*}
%\Theta(A,B,C) & = & -8640A^2-655360A^3-245760A^2B-92160A^2C \\
%              & + & 3840AB^2+ 5760ABC-1280B^3-163840A^2C^2 \\
%				      & + & 184320AB^2C+ 51200ABC^2-1280AC^3-34560B^4\\
%							& - & 11520B^3C+320B^2C^2-10240AC^4+2560B^2C^3.
%\end{eqnarray*}
%
%\textcolor{blue}{Pode simplificar $\Theta(A,B,C)$ dividindo por um factor comun dos coeficientes. Tambem pode escrever a expressão como um polinomio de A.}
\newpage
\begin{table}[ht!!!]
\caption{Singularities of the folding map $f_{\eta}$ when $\eta$ transverse to the tangent cone.}
\begin{center}
{\begin{tabular}{llcc}
\hline
Type & Normal form     & $\mathcal{A}_e$-cod.&  Conditions\\
\hline
$B^{\pm}_{2}$ & $(x,y^2,x^2y\pm y^5)$     & $2$ & \mbox{$\alpha\neq0$, $\gamma\neq p_3\alpha$}                             \\
$B^{\pm}_{3}$ & $(x,y^2,x^2y\pm y^7)$     & $3$ & \mbox{$\alpha\neq0$, $\gamma=p_3\alpha$}, $(*)$\\
$B^{\pm}_{4}$ & $(x,y^2,x^2y\pm y^9)$     & $4$ & \mbox{$\alpha\neq0$, $\gamma=p_3\alpha$, $(**)$}\\
$C^{\pm}_{3}$ & $(x,y^2,xy^3\pm x^3y)$    & $3$ & \mbox{$\alpha=0$, $\Phi(\beta,\gamma)\neq0$}                            \\
$C^{\pm}_{4}$ & $(x,y^2,xy^3\pm x^4y)$    & $4$ & \mbox{$\alpha=0$, $\Phi(\beta,\gamma)=0$}                               \\
$F_{1,0}$     & $(x,y^2,x^3y+A_1xy^5+B_1y^7)$ & $4$ & \mbox{$\beta=1$, $4A_1^3+27B_1^2\neq0$}                                    
\\
\hline
\end{tabular}
}
\end{center}
\label{table1}
\end{table}
where $\Phi(\beta,\gamma)=-2b\beta^3+(4a-b^2+2)\beta^2\gamma+\gamma^3$ and $F_{1,0}$ is a unimodal singularity (see \cite[page 29]{Mond1982}).\\

\noindent {\footnotesize
$(*)$ The singularity $B^{\pm}_{3}$ occurs for generic $\eta$ on the plane curve $\gamma=p_3\alpha$ in $\mathbb{S}^2$.\\
$(**)$ The singularity $B^{\pm}_{4}$ occurs for special values of $\eta$ on the curve $\gamma=p_3\alpha$ 
whose expression too lengthy to reproduce here, see \cite{Barajas2017} for details}.\\

\begin{small}
\begin{table}[ht!!!]
\caption{Singularities of the folding map $f_{\eta}$ when $\eta$ is in the tangent cone but is not parallel to the tangential direction.}
\begin{center}
{
\begin{tabular}{llcc}
\hline
Type & Normal form     & $\mathcal{A}_e$-cod.&  Conditions\\ \hline
$P_{3}$              & $(x,xy+y^3,xy^2+ky^4)$              & $3$  & \mbox{$\alpha\neq -p_3\gamma$, $k\neq\frac{1}{2},1,\frac{3}{2}$} \\
$P_{4}(\frac{1}{2})$ & $(x,xy+y^3,xy^2+\frac{1}{2}y^4)$    & $4$  & \mbox{$\alpha\neq -p_3\gamma$, $\Psi(\frac{1}{2})=0$}    \\
$P_{4}(\frac{3}{2})$ & $(x,xy+y^3,xy^2+\frac{3}{2}y^4)$    & $4$  & \mbox{$\alpha\neq -p_3\gamma$, $\Psi(\frac{3}{2})=0$}    \\
$P_{4}(1)$           & $(x,xy+y^3,xy^2+y^4)$               & $4$  & \mbox{$\alpha\neq -p_3\gamma$, $\Psi(1)=0$}\\              
$R_{4}$              & $(x,xy+y^6+A_2y^7,xy^2+y^4+B_2y^6)$ & $4$  & \mbox{\hspace{0.7cm}$\alpha=-p_3\gamma$} \\
$T_{4}$              & $(x,xy+y^3,y^4)$                    & $4$  & $\gamma=1$    \\                          
\hline
\end{tabular}
}
\end{center}
\label{table2}
\end{table}
\end{small}

\noindent where $\Psi(k)=\Psi(k,\alpha,\gamma)=-2k\alpha(\alpha+p_3\gamma)+1$.
\end{thm}

\begin{proof}
The proof follows by making successive changes of coordinates on the jet level and using the 
conditions for a map-germ to be $\mathcal A$-equivalent to one in Mond's list. When $\eta$ is parallel to the tangential direction, the condition $\Theta(A,B,C)\ne 0$ is for the germ \eqref{eq_corank2} to be 3-$\mathcal A_1$-determined. Details of the calculations can be found in \cite{Barajas2017}. The conditions for the folding map $f_{\eta}$ to have the different types of singularities listed are given on the vector $\eta$, therefore the result is true for an open set in the space of geometric cross-caps.
\end{proof}
\begin{prop}
The non-stable singularities of $f_{\eta}$ are not  $\mathcal{A}_e$-versally unfolded by the family $F$.
\end{prop}
\begin{proof}
Fix $\eta_0=\left(\alpha_0,\beta_0,\gamma_0\right)\in\mathbb{S}^2$ and consider o map germ $f_{(\eta_0,0)}$. Let $\xi=\left(u,v,w\right)\in\mathbb{R}^3$ and consider a real number $\delta\neq0$, $\left|\delta\right|<\varepsilon$ ($\varepsilon$ small). Thus, $f_{\left(\bar{\eta},\delta\right)}$ with $\bar{\eta}=\frac{\eta_0+\xi}{\left\|\eta_0+\xi\right\|}\in\mathbb{S}^2$ is a deformation of $(\eta_0,0)$. Let $F(x,y,u,v,w,\delta)=\left(f(x,y,u,v,w,\delta),u,v,w,\delta\right)$, where $f(x,y,u,v,w,\delta)=f_{\left(\bar{\eta},\delta\right)}(x,y)$ and 
$$f(x,y,0,0,0,0)=f_{(\eta_0,0)}(x,y).$$
It follows that $F$ is a 4-parameter unfolding of the map germ $f_{(\eta_0,0)}$. By Proposition \ref{prop3.1}, the map germ $f_{(\bar{\eta},\delta)}$ has singularity type cross-cap at origin. By changes of coordinates in the source and target, the map germ $j^2f_{\left(\eta_0,0\right)}$ is $\mathcal{A}$-equivalent to
\begin{equation}\label{2jetfeta0}
j^2f_{\left(\eta_0,0\right)}\sim_{\mathcal{A}} \left(x,-\gamma_0\left(\gamma_0-b\beta_0\right)xy+\beta_0\gamma_0y^2,\beta_0\left(\gamma_0-b\beta_0\right)xy-\beta_0^2y^2\right).
\end{equation} 
We consider the 2-jet of $f_{\left(\eta_0,0\right)}$ since 
 the map germ $f_{\left(\bar{\eta},\delta\right)}$, with $\delta\neq0$, has a cross-cap singularity. We take, without loss of generality,  
 $f_{\left(\eta_0,0\right)}$ as \eqref{2jetfeta0}. Thus, the initial speeds of $F$ are given by
\begin{align*} 
\dot{F}_u(x,y) =& \left(0, 2\alpha_0xy, 2b\alpha_0xy+2\alpha_0y^2\right), \\ 
\dot{F}_v(x,y)  =&\left(0, (b\gamma_0+2\beta_0)xy+\gamma_0y^2, \gamma_0xy\right), \\
\dot{F}_w(x,y) =& \left(0, b\beta_0xy+\beta_0y^2, (2b\gamma_0+\beta_0)xy+2\gamma_0y^2\right), \\
\dot{F}_\delta(x,y) =& \left(0,(2b\beta_0\gamma_0+2\left(1-\gamma_0^2\right))xy+2\gamma_0\beta_0y^2,\right. \\
                    & \left.2\left(\beta_0\gamma_0+b\left(1-\beta_0^2\right)\right)xy+2\left(1-\beta_0^2\right)y^2\right).
\end{align*}
Note that the Jacobian ideal of $j^2f_{\left(\eta_0,0\right)}$ is generated by
\begin{align*} 
\left(j^2f_{\left(\eta_0,0\right)}\right)_x =& \left(1,-\gamma_0\left(\gamma_0-b\beta_0\right)y,\beta_0\left(\gamma_0-b\beta_0\right)y\right),\\
\left(j^2f_{\left(\eta_0,0\right)}\right)_y =& \left(0,-\gamma_0\left(\gamma_0-b\beta_0\right)x+2\beta_0\gamma_0y,\beta_0\left(\gamma_0-b\beta_0\right)x-2\beta_0^2y\right).
\end{align*}
It follows that
\begin{align*} 
\left(j^2f_{\left(\eta_0,0\right)}\right)_x &- \left(j^2f_{\left(\eta_0,0\right)}\right)^*\left(1\right)\cdot e_1= \left(0,-\gamma_0\left(\gamma_0-b\beta_0\right)y,\beta_0\left(\gamma_0-b\beta_0\right)y\right),\\
%\end{equation*} 
%
%\begin{equation*} 
\left(j^2f_{\left(\eta_0,0\right)}\right)_y &+ \left(\gamma_0-b\beta_0\right)\left(j^2f_{\left(\eta_0,0\right)}\right)^*\left(x\right)\cdot\left(\gamma_0 e_2-\beta_0 e_3\right) = \left(0,2\beta_0\gamma_0y,-2\beta_0^2y\right).
\end{align*} 

It follows that it is not possible to simultaneously obtain the vectors $\left(0,y,0\right)$ and $\left(0,0,y\right)$ in the extended tangent space $T_e\mathcal{A}\left(j^2f_{\left(\eta_0,0\right)}\right)$. Therefore $T_e\mathcal{A}\left(j^2f_{\left(\eta_0,0\right)}\right)+\mathbb{R}\left\{\dot{F}_u,\dot{F}_v,\dot{F}_w,\dot{F}_\delta\right\}\neq\mathcal{E}(2,3)$, which concludes the proof.
\end{proof}
%
%
%%%%%%%%%%%%%%%%%%%%%%%%%%%%%%%%%%%%%%%%%%%%%%%%%%%%%%%%%%%%%%%%%%%%%%%%%%%
%%%%  The geometry of the folding maps on a cross-cap
%%%%%%%%%%%%%%%%%%%%%%%%%%%%%%%%%%%%%%%%%%%%%%%%%%%%%%%%%%%%%%%%%%%%%%%%%%%
%
\section{The geometry of the folding maps on a cross-cap}\label{section4}
In this section we characterize geometrically some singularities that occur in the members of the family of folding maps on a cross-cap.
We use the concepts in \S \ref{section2} and the maps which we define below.

Let $\theta(t)=(t,\alpha(t))$ be a germ of a regular curve with $\alpha'(0)=a_1\ne 0$ %\textcolor{blue}{O que acontece se $a_1=0$?}
and consider its image $\mu(t)=\phi\circ\theta(t)$ on the cross-cap. Note that $\mu$ is a regular curve and $\mu'(0)$ is parallel to the tangential line. When $\alpha'(0)=0$ the image on the cross-cap is a singular curve. Curves with this condition will be considered later.

Let $N(p)$ denote the unit normal vector to the cross-cap surface away from the cross-cap point. At the cross-cap point, the surface has no well defined normal vector, that is, $N(p)$ does not extend to the cross-cap point. However, a simple calculation shows that
\begin{equation*}
\lim_{t\rightarrow0}N\left(\mu(t)\right)=\left(0,\frac{-\left(b+2a_1\right)}{\left(\left(b+2a_1\right)^2+1\right)^{\frac{1}{2}}},\frac{1}{\left(\left(b+2a_1\right)^2+1\right)^{\frac{1}{2}}}\right).
\end{equation*}
Therefore, the normal vector has a limiting direction along the curve $\mu$ at the cross-cap point, and this limiting direction is parallel to $(0,-(b+2a_1),1)$. Observe that the limiting direction depends only on $\theta'(0)$ and not on the curve $\mu$. Thus, we can define the following map.
\begin{definition}
Let $X$ be a cross-cap as in \eqref{Par_Cros} and let $N_0X$ denote its normal plane at the origin. We call the {\rm limiting normal  map} of $X$ the map $LN:\mathbb{R}^2\rightarrow N_0X$ given by
\begin{equation*}
LN(v_1,v_2)=\left(0,-(bv_1+2v_2),v_1\right).
\end{equation*}
\end{definition}
The limiting normal map induces a bijection between $\mathbb{S}^1\subset\mathbb{R}^2$ and $\mathbb{S}^1\subset N_{(0,0,0)}X$. 

Consider now the family of curves $\varpi_{\lambda}(t)=\left(\beta_{\lambda}(t),t\right)$, with 
$\beta_{\lambda}(0)=\beta_{\lambda}'(0)=0$ and  $\beta_{\lambda}''(0)=\lambda$. We have $\varpi'_{\lambda}(0)=(0,1)$ and
$
LN(0,1)=\left(0,-2,0\right). 
$
 (The direction $LN(0,1)$ is orthogonal to the tangent cone of the cross-cap.) Thus, the limiting normal direction to the cross-cap is the same for all the members of 
the family of curves $\varpi_{\lambda}$.  The image of $\varpi_{\lambda}(t)$ by $\phi$ is a singular curve on the cross-cap 
and its limiting tangent direction belongs to the tangent cone $TC_0X$ of the cross-cap. 
More precisely,
$$
\lim_{t\rightarrow0}\left(\phi\circ\varpi_{\lambda}\right)'(t)=\left(2\lambda,0,2\right).
$$

\begin{definition}
We call {\rm the limiting tangent map} of the cross-cap as in \eqref{Par_Cros} the map 
$LT:\mathbb{R}^2\rightarrow TC_0X$ given by
$$
LT(\lambda,1)=\frac{1}{2}\lim_{t\rightarrow0}\left(\phi\circ\varpi_{\lambda}\right)'(t)=
\left(\lambda,0,1 \right).
$$
\end{definition}

We use the limiting tangent and normal maps of the cross-cap to characterise the singularities of the folding maps.

\begin{thm}\label{th2}
Let $X$ be a cross-cap parametrised as in \eqref{Par_Cros}. 
We have the following characterization of the singularities of the folding maps $f_{\eta}$ on $X$.
\begin{enumerate}

	\item[(i)] The $C_4$ singularity occurs when folding with respect to a plane generated 
	by the tangential direction of the cross-cap and $LN(w)$, where $w$ is a tangent 
	direction at the origin to a sub-parabolic curve.

	\item[(ii)] The $F_{1,0}$ and $T_4$ singularities occur when folding with respect to 
	the plane generated by the tangential direction of the cross-cap and 
	$LN(w)$, where the $w$  is a tangent 
		direction at the origin to a ridge curve relative to $\nu_2$. %{\bf relative to $v_2$ ?? precisa disso}
		In particular $f_\eta$ has a singularity type $F_{1,0}$ when $\eta$ is orthogonal to the tangent cone. %{\bf n\~ao \'e isso que o Teorema 3.3 disse...}

	\item[(iii)] The $P_4(\frac{1}{2})$ singularity occurs when folding with respect 
	to the plane generated by $LN(0,1)$ and the limiting tangent direction to the double point curve.
	\item[(iv)] The $P_4(\frac{3}{2})$ singularity occurs when folding with respect to the plane 
	generated by $LN(0,1)$ and the limiting tangent direction to the separatrix of principal foliations.
	\item[(v)] When $\eta$ is parallel to the limiting tangent direction of the double point curve, the singularity of $f_{\eta}$ is of type  $R_4$.
	\item[(vi)] The corank 2 map singularity of $f_{\eta}$ 
	occur when $\eta$ is parallel to the tangential direction, i.e., the folding plane is the normal plane.

\end{enumerate}
\end{thm}

\begin{proof}
	%\textcolor{blue}{Na demonstração voce mostra que se tem condi\c{c}\~oes geometricas então algum tipo de singularidade acontece. No enunciado voce escreve o contrario, a singularidade acontece aquando tem as condições geometricas....}

We will characterize, in each case, the normal direction $\eta$ to the folding plane.

(i) Consider the map germ $f_\eta$ such that $\eta=(0,\beta,\gamma)$ satisfying $\Phi(\beta,\gamma)=0$ as in the Theorem \ref{th1}. By direct calculations we have that 
\begin{equation*}
(LN)^{-1}((0,-\gamma,\beta))=(w_1,w_2),
\end{equation*}
 where $\lambda,\mu$ satisfies
\begin{equation*}
abw_1^3+(b^2+2a+1)w_1^2w_2+3bw_1w_2^2+2w_2^3=0,
\end{equation*}
which is the same condition in the equation \eqref{eq_subparab}.
%(i) Consider the map germ $f_\eta$ such that $\eta$ is orthogonal to the plane $W$ generated by the tangent line and $LN(w)$, where $w=(w_1,w_2)$ is a tangent direction at the origin to a sub-parabolic curve. Then $w$ satisfies the equation \eqref{eq_subparab} and applying the limiting normal map in $w$ we have that $LN(w)=(\sigma,\rho)$ satisfying
%
%\begin{equation}\label{eq_nplane}
%\sigma^3+\left(4a-b^2+1\right)\sigma\rho^2+2b\rho^3=0.
%\end{equation}

%Given that the tangent line belong to $W$, $\eta$ belong to the normal plane and it is orthogonal to $LN(w)$, thus $\eta$ is parallel to $(\rho,-\sigma)$. It follows that if $\eta$ satisfies
%
%\begin{equation*}
%-4\left(-2b\sigma^3+\left(4a-b^2+1\right)\sigma^2\rho+\rho^3\right)=\Phi(\sigma,\rho)=0,
%\end{equation*}
%	
%for $\Phi$ as Theorem \ref{th1}, then $f_\eta$ has a singularity at the origin of type $C_4$.

(ii) Take the map germ $f_\eta$ such that $\eta=(0,1,0)$, thus by the Theorem \ref{th1}, $f_\eta$ has a $F_{1,0}$ singularity at the origin. Follows that
\begin{equation*} 
(LN)^{-1}(0,0,1)=\left(1,-\frac{b}{2}\right)=(w_1,w_2)\quad\text{satisfying}\quad bw_1+2w_2=0,
\end{equation*}
which satisfies the equation \eqref{eq_ridge1}. Now consider the map germ $f_\eta$ such that $\eta=(0,0,1)$ and again by the Theorem \ref{th1}, $f_\eta$ has a $T_{4}$ singularity at the origin. Using the same idea
\begin{equation*} 
(LN)^{-1}(0,1,0)=\left(0,-\frac{1}{2}\right)=(w_1,w_2)\quad\text{satisfying}\quad w_1=0,
\end{equation*}
which is the other one condition that satisfy the equation \eqref{eq_ridge1}.

%(ii) Take the map germ $f_\eta$, where $\eta$ is orthogonal to the plane $W$ generated by the tangent line and $LN(w)$, where $w=(w_1,w_2)$ is a tangent direction at the origin to a ridge curve. Then $w$ satisfies the equation \eqref{eq_ridge1}, this is, $w$ is parallel to $(0,1)$ or $(1,-\frac{b}{2})$. Appliying the limiting normal map in $(0,1)$ we have that $LN(0,1)=(0,-2,0)$, thus $\eta$ is parallel to $(0,0,1)$ and this lies with the case $\gamma=1$ in Theorem \ref{th1}. Appliying now the limiting normal map in $(1,-\frac{b}{2})$ we have $LN(1,-\frac{b}{2})=(0,0,1)$, thus $\eta$ is parallel to $(0,1,0)$ and this lies with the case $\beta=1$ in Theorem \ref{th1}.

(iii) Consider the map germ $f_\eta$ such that $\eta=(\alpha,0,\gamma)$ satisfying $\Psi\left(\frac{1}{2}\right)=0$ and $\alpha\neq-p_3\gamma$. Note that
\begin{equation}\label{eq_psi12}
\Psi\left(\frac{1}{2}\right)=\Psi\left(\frac{1}{2},\alpha,\gamma\right)=-\alpha\gamma p_3+(1-\alpha^2)=\gamma(\gamma-p_3\alpha).
\end{equation}
Again we consider the orthogonal a $\eta$ in the tangent cone and applying to the equation \eqref{eq_psi12} and equaling to zero we obtain
\begin{equation*}
\Psi\left(\frac{1}{2},-\gamma,\alpha\right)=\alpha(\alpha+p_3\gamma)=0.
\end{equation*}
As $\alpha\neq0$ follows that $\alpha+p_3\gamma=0$. Calculating $(LT)^{-1}(-p_3,0,1)=(-p_3,1)$, which correspond with the 2-jet of curve of double points of the cross-cap $x=-p_3y^2$.  

%(iii)  Let $f_\eta$ the map germ where $\eta$ is orthogonal to the plane $W$ generated by $NL(0,1)$ and $TL(-p_3,1)=(-p_3,0,1)$ (see Section \ref{section2}). Thus $\eta$ is parallel to $(1,0,p_3)$ and $\Psi(\frac{1}{2})=\gamma(\gamma-p_3\alpha)$, therefore by Theorem \ref{th1} the germ $f_\eta$ has a singularity at the origin of type $P_4\left(\frac{1}{2}\right)$.

(iv) Analogous as (iii) take a map germ $f_\eta$ such that $\eta=(\alpha,0,\gamma)$ satisfying $\Psi\left(\frac{3}{2}\right)=0$ and $\alpha\neq-p_3\gamma$. It follows that
\begin{equation}\label{eq_psi32}
\Psi\left(\frac{3}{2}\right)=\Psi\left(\frac{3}{2},\alpha,\gamma\right)=-3\alpha\gamma p_3-3\alpha^2+1=-3\alpha\gamma p_3-2\alpha^2+\gamma^2.
\end{equation}
As before we consider the orthogonal a $\eta$ in the tangent cone and applying to the equation \eqref{eq_psi32} we obtain
\begin{equation}\label{eq_psi32lambda}
\Psi\left(\frac{3}{2},-\gamma,\alpha\right)=3\alpha\gamma p_3+\alpha^2-2\gamma^2=\lambda^2+3p_3\lambda-2=0,
\end{equation}
where $\lambda=\frac{\alpha}{\gamma}$. Via $(LT)^{-1}$, the vectors $(-\gamma,0,\alpha)$ in the tangent cone satisfying the equation \eqref{eq_psi32lambda} correspond with the curves $x=\lambda y^2$ where $\lambda$ satisfies the same equation above. (Compare with the equation \eqref{eq_sep_lc}).

%(iv) Consider $f_{\eta_i}$, $i=1,2,$ the map germ where $\eta_i$ is orthogonal to the plane $W_i$ generated by $NL(0,1)$ and $TL(\lambda_i,%1)=(\lambda_i,0,1)$, where $\lambda_i$ is given by equation \eqref{eq_sep_lc}. It follows that $\eta_i$ is parallel to $(-1,0,\lambda_i)$ %and note that 
%
%\begin{equation}\label{psi32}
%\Psi\left(\frac{3}{2}\right)=-2\alpha^2-3\alpha\gamma p_3+\gamma^2.
%\end{equation}
%
%Without loss of generality we can suppose that $\eta_i=(-1,0,\lambda_i)$, thus for $\eta_i=(\alpha_i,\beta_i,\gamma_i)$, $\gamma_i=-\lambda%_i\alpha_i$. Then  \eqref{psi32} gives
%
%\begin{equation*}
%\alpha_i^2(\lambda_i^2+3p_3\lambda_i-2)=0,\quad\text{with $\alpha_i\neq0$.}
%\end{equation*}
%
%It follows from Theorem \ref{th1} that $f_{\eta_i}$, $i=1,2,$ has a singularity at the origin of type $P_4\left(\frac{3}{2}\right)$.

(v) As $\eta$ is tangent to $(-p_3,0,1)$, in particular satisfies $\alpha=-p_3\gamma$, which is exactly the condition in the Theorem \ref{th1} for the germ $f_\eta$ to have a singularity at the origin of type $R_4$.

%(v) As in the item (iii), $TL(-p_3,1)=(-p_3,0,1)$. Without loss of generality we take $\eta=(-p_3,0,1)$, thus for $\eta=(\alpha,\beta,\gamma)$ we have $\alpha=-p_3\gamma$, which is the condition in Theorem \ref{th1} for the germ $f_\eta$ to have a singularity at the origin of type $R_4$.

(vi)  In this case $\eta=(1,0,0)$ and by definition, $f_\eta$ is the folding map respect to the normal plane. The rest follows from Theorem \ref{th1}.
\end{proof}

\textit{Acknowledgements:} The results in this paper are a part of the author's Ph.D. thesis. He would like to thank 
Farid Tari for his help, advice and friendship. The author was supported by a CAPES doctoral grant.

%%%%%%%%%%%%%%%%%%%%%%%%%%%%%%%%%%%%%%%%%%%


\nocite{*}
\begin{thebibliography}{}
%   1
\bibitem{Arnold1985} V. I. Arnold and S. M. Gusein-Zade and A. N. Varchenko.: Singularities of Differential Maps I. Classification of Critical Points, Caustics and Wave Fronts. Birkh\"{a}user, (1985).
%   2
\bibitem{Barajas2017} Barajas S., M.: Sobre a geometria diferencial do cross-cap no 3-espa\c co {E}uclidiano (in portuguese). PhD thesis, Instituto de Ci\^encias Matem\'aticas e de Computa\c c\~ao, Universidade de S\~ao Paulo (ICMC-USP), 2017.
%   3
\bibitem{BarajasKabata2016} Barajas S., M. and Kabata Y.: Projection of crosscap. International Journal of Geometric Methods in Modern Physics, World Scientific Pub Co Pte Lt, (2019).
%   4
\bibitem{Bruce1984} Bruce, J. W.: Projections and reflections of generic surfaces in $\mathbb{R}^3$. Math. Scand., 54, 262-278, (1984).
%   5 
\bibitem{Bruce1987} Bruce, J. W., DuPlessis, A. A. and  Wall, C. T. C.: Determinacy and Unipotency. Invent. Math, 88 521-554,(1987). 
%   6
\bibitem{BruceWest1998} Bruce, J. W. and  West, J. M.: Functions on a crosscap. Math. Proc. Camb. Phil. Soc., 123, 19-39, (1998).
%   7
\bibitem{Bruce1991} Bruce, J. W. and Wilkinson, T. C.: Folding maps and focal sets. Lecture Notes in Mathematics, Springer-Verlag,1462, 63-72, (1991).
%   8
\bibitem{Dias2016} Dias, F. C. and Tari F.: On the geometry of the cross-cap in {M}inkowski 3-space and binary differential equations. Tohoku Mathematical Journal, 68-2, 293-328, (2016).
%   9
\bibitem{Fukui2012} Fukui, F. and Hasegawa, M.: Fronts of {W}hitney umbrella - A differential geometric approach via blowing up. Journal of Singularities, 4, 35-67, (2012).
%   10
\bibitem{Garcia2000} Garcia, R., Gutierrez, C. and Sotomayor, J.: Lines of principal curvature around umbilics ans {W}hitney umbrellas. Tohoku Mathematical Journal, 52, 163-172, (2000).
%   11
\bibitem{livro_cidinha_farid_2015} Izumiya, S., Romero, C., Ruas, M. A. S. and Tari, F.: Differential Geometry from a Singularity Theory Viewpoint. World Scientific Publishing Company, (2016).
%   12
\bibitem{Izumiya2010} Izumiya, I., Takahashi, M. and Tari F.: Folding maps on spacelike and timelike surfaces and duality. Osaka Journal of Mathematics, 47-3, 839-862, (2010).
%   13
\bibitem{Martinet1982} Martinet, J.: Singularities of smooth functions and maps. LMS Lecture Notes 58, Cambridge University Press, (1982).
%   14
\bibitem{Mond1982} Mond, D.: The classification of germs of maps from surfaces to 3-space, with applications to the differencial geometry of immersions. PhD thesis, University of Liverpool, (1982).
%   15
\bibitem{Mond1985} Mond, D.: On the classification of germ of maps from $\mathbb{R}^2$ to $\mathbb{R}^3$. Proc. London Math. Soc., s3-50 (2), 333-369, (1985).
%   16
\bibitem{Morris1996} Morris,R.: The sub-parabolic lines of a surface. Mathematics of Surfaces VI, Ed. Glen Mullineux, IMA new series 58, Clarendon Press, Oxford, 79-102, (1996).
%   17
\bibitem{Porteus1983} Porteous, I. R.: The normal singularities of surfaces in $\mathbb{R}^3$. Proceedings of Symposia in Pure Mathematics, 40, 379-393, (1983).
%   18
\bibitem{Tari2007} Tari, F.: On pairs of geometric foliations on a cross-cap. Tohoku Mathematical Journal, 59, 233-258, (2007).
%   19
\bibitem{West1995} West, J. M.: The differential geometry of the crosscap. PhD thesis, University of Liverpool, (1995).
%   20
\bibitem{Whitney1943} Whitney, H.: The general type of singularity of a set of 2n-1 smooth functions of n variables. Duke Math Journal, 10, 161-172, (1943).
%   21
\bibitem{Wilkinson1991} Wilkinson, T. C.: The geometry of folding maps. PhD thesis, University of Newcastle upon Tyne, (1991).
%
\end{thebibliography}
\end{document}